\documentclass[12pt]{article}

\usepackage{latexsym,amssymb,upref,amsmath,amsthm, amsfonts,authblk}
\usepackage{amssymb,amsmath,amsthm, calc, graphicx}
\usepackage{epsfig}
\usepackage{breqn}
\usepackage{footnpag}
\usepackage{rotating}
\usepackage{amsfonts}
\usepackage{setspace}
\usepackage{fullpage}
\usepackage{enumitem}
\usepackage{bbold}
\usepackage{comment}
\usepackage{pgf,tikz}
\usepackage{mathrsfs}
\usetikzlibrary{arrows}
\usepackage{yhmath}
\usepackage{hyperref}
\usepackage{authblk}

\newtheorem{thm}{Theorem}%[section]
\newtheorem{definition}{Definition}
\newtheorem{claim}{Claim}
\newtheorem{lemma}[thm]{Lemma}

\newtheorem{notation}[thm]{Notation}

\newtheorem{prop}[thm]{Proposition}
\newtheorem{observation}[thm]{Observation}

\newcommand{\thistheoremname}{}
\newtheorem*{genericthm*}{\thistheoremname}
\newenvironment{namedthm*}[1]
{\renewcommand{\thistheoremname}{#1}%
	\begin{genericthm*}}
	{\end{genericthm*}}

\newcommand{\abs}[1]{\left\lvert{#1}\right\rvert}
\newcommand{\floor}[1]{\left\lfloor{#1}\right\rfloor}
\newcommand{\ceil}[1]{\left\lceil{#1}\right\rceil}

\title{On the Rainbow Tur\'an number of paths}

\author
{
	Beka Ergemlidze
	\thanks{ Department of Mathematics, Central European University, Budapest.
		E-mail: \texttt{beka.ergemlidze@gmail.com}} \qquad 
	Ervin Gy\H{o}ri \thanks{R\'enyi Institute, Hungarian Academy of Sciences and 
		Department of Mathematics, Central European University, Budapest. E-mail: \texttt{gyori.ervin@renyi.mta.hu}} \qquad 
	Abhishek Methuku \thanks{Department of Mathematics, Central European University, Budapest. (Corresponding) E-mail: \texttt{abhishekmethuku@gmail.com}} 
}

\begin{document}
	
	\maketitle
	
	\begin{abstract}
		Let $F$ be a fixed graph. The \emph{rainbow Tur\'an number} of $F$ is defined as the maximum number of edges in a graph on $n$ vertices that has a proper edge-coloring with no rainbow copy of $F$ (where a rainbow copy of $F$ means a copy of $F$ all of whose edges have different colours). The systematic study of such problems was initiated by Keevash, Mubayi, Sudakov and Verstra\"ete.
		
		In this paper, we show that the rainbow Tur\'an number of a path with $k+1$ edges is less than $\left(\frac{9k}{7}+2\right) n$, improving an earlier estimate of Johnston,  Palmer and Sarkar.
	\end{abstract}

	\section{Introduction}
	
Given a graph $F$, the maximum number of edges in a graph on $n$ vertices that contains no copy of $F$ is known as the \emph{Tur\'an number} of $F$, and is denoted by $ex(n, F)$. An edge-colored graph is called \emph{rainbow} if all its edges have different colors.  Given a graph $F$, the \emph{rainbow Tur\'an number} of $F$ is defined as the maximum number of edges in a graph on $n$ vertices that has a proper edge-coloring with no rainbow copy of $F$, and it is denoted by $ex^{*}(n, F)$. 

The systematic study of rainbow Tur\'an numbers was initiated in \cite{KMSV} by Keevash, Mubayi, Sudakov and Verstra\"ete. Clearly, $ex^{*}(n, F) \ge ex(n,F)$. They determined $ex^{*}(n, F)$ asymptotically for any non-bipartite graph $F$, by showing that $ex^{*}(n, F) = (1+o(1)) ex(n,F)$.
For bipartite $F$ with a maximum degree of $s$ in one of the parts, they proved $ex^{*}(n, F) = O(n^{1/s})$. This matches the upper bound
for the (usual) Tur\'an numbers of such graphs. 

Keevash, Mubayi, Sudakov and Verstra\"ete also studied the rainbow Tur\'an problem for even cycles. More precisely, they showed that $ex^{*}(n, C_{2k}) = \Omega(n^{1+1/k})$ using the construction of large $B^{*}_k$-sets of Bose and Chowla \cite{Bose_Chowla}-- it is conjectured that the same lower bound holds for $ex^{*}(n, C_{2k})$ and is a well-known difficult open problem in extremal graph theory. They also proved a matching upper bound in the case of six-cycle $C_6$, so it known that $ex^{*}(n, C_{6}) = \Theta(n^{4/3}) = ex(n, C_6)$. However, interestingly, they showed that $ex^{*}(n, C_{6})$ is asymptotically larger than $ex(n, C_{6})$ by a multiplicative constant. 
Recently,  Das, Lee and Sudakov \cite{Das} showed that $ex^{*}(n, C_{2k}) = O(n^{1+\frac{(1+\epsilon_k) \ln k}{k}})$, where $\epsilon_k \to 0$ as $k \to \infty$. 

For an integer $k$, let $P_k$ denote a path of length $k$, where the length of a path is defined as the number of edges in it. Erd\H{o}s and Gallai \cite{ErdGallai} proved that $ex(n,P_{k+1}) \le  \frac{k}{2}n$; moreover, they showed that if $k+1$ divides $n$, then the unique extremal graph is the vertex-disjoint union of $\frac{n}{k+1}$ copies of $K_{k+1}$. 

%Keevash, Mubayi, Sudakov and Verstra\"ete \cite{KMSV} noted that  determining the asymptotic behaviour of $ex^{*}(n, P_{k+1})$ is an interesting problem, and noted that 

On the other hand, Keevash, Mubayi, Sudakov and Verstra\"ete \cite{KMSV} showed that in some cases, the rainbow Tur\'an number of $P_k$ can be strictly larger than the usual Tur\'an number of $P_k$: Maamoun and Meyniel \cite{Maamoun} gave an example of a proper coloring of $K_{2^k}$ containing no rainbow path with $2^k-1$ edges. By taking a vertex-disjoint union of such $K_{2^k}$'s, Keevash ~et. ~al. showed that $ex^{*}(n, P_{2^k-1}) \ge \binom{2^k}{2} \floor{\frac{n}{2^k}} = (1+o(1)) \frac{2^k-1}{2^k-2} ex(n, P_{2^k-1})$-- so $ex^{*}(n, P_{2^k-1})$ is not asymptotically equal to $ex(n, P_{2^k-1})$. They also mentioned that determining the asymptotic behavior of $ex^{*}(n, P_{k+1})$ is an interesting open problem, and stated the natural conjecture that the optimal construction is a disjoint union of cliques of size $c(k)$, where $c(k)$ is chosen as large as possible so that the cliques can be properly colored with no rainbow $P_{k+1}$. For  $P_4$, this conjecture was disproved by Johnston,  Palmer and Sarkar \cite{JPalmer}: Since any properly edge-colored $K_5$ contains a rainbow $P_4$, and $K_4$ does not contain a $P_4$, the conjecture for $P_4$ would be that $ex^{*}(n, P_{4}) \sim \frac{3n}{2}$. But they show that in fact, $ex^{*}(n, P_{4}) \sim 2n$ by showing a proper edge-coloring of $K_{4,4}$ without no rainbow $P_4$, and then taking $\frac{n}{8}$ vertex-disjoint copies of $K_{4,4}$. For general $k$, they proved the following:

\begin{thm}[Johnston,  Palmer and Sarkar \cite{JPalmer}]
For any positive integer $k$, we have
$$\frac{k}{2}n \le ex^{*}(n, P_{k+1})  \le  \ceil{\frac{3k+1}{2}} n.$$
\end{thm}

%For a graph $G$, let the degree of a vertex $v \in V(G)$ be denoted by $d(v)$. 

%Length of a path is the number of edges in it. For any integer $k$, let $P_k$ denote a path of length $k$.
	
	We improve the above bound by showing the following: 
	
	\begin{thm}
		\label{main_thm}
		For any positive integer $k$, we have
		$$ex^{*}(n, P_{k+1}) < \left(\frac{9k}{7}+2\right)n.$$
	\end{thm}

We remark that using the ideas introduced in this paper, it is conceivable that the upper bound may be further improved. However, it would be very interesting (and seems to be difficult) to prove an upper bound less than $kn$.
	\vspace{2mm}
	
We give a construction which shows that $ex^{*}(n, P_{2^k}) > ex^{*}(n, P_{2^k})$ for any $k \ge 2$. 

%For $k = 2$, this construction coincides with the example of Johnston,  Palmer and Sarkar \cite{JPalmer}. For $k \ge 3$, this construction potentially) provides counter examples to the above mentioned conjecture of Keevash, Mubayi, Sudakov and Verstra\"ete \cite{KMSV}. 

\vspace{2mm}

\textbf{Construction. } Let us first show a proper edge-coloring of $K_{2^k, 2^k}$ (a complete bipartite graph with parts $A$ and $B$, each of size $2^k$) with no rainbow $P_{2^k}$. The vertices of $A$ and $B$ are both identified with the vectors $\mathbb{F}_2^k$.  Each edge $uv$ with $u \in A$ and $v \in B$ is assigned the color $c(uv) := u - v$.  Clearly this gives a proper edge-coloring of $K_{2^k, 2^k}$. Moreover, if it contains a rainbow path $v_0v_1 \ldots v_{2^k}$ then such a path must use all of the colors from $\mathbb{F}_2^k$. Therefore $\sum_{i=0}^{2^k-1} c(v_i v_{i+1}) = 0$. On the other hand,  $\sum_{i=0}^{2^k-1} c(v_i v_{i+1}) = \sum_{i=0}^{2^k-1} (v_i-v_{i+1}) = v_0 - v_{2^k}$. Thus, $v_0 - v_{2^k} = 0$. But notice that since the length of the path $v_0v_1 \ldots v_{2^k}$ is even, its terminal vertices $v_0$ and $v_{2^k}$ are either both in $A$ or they are both in $B$. So they could not have been identified with the same vector in $\mathbb{F}_2^k$, a contradiction. Taking a vertex-disjoint union of such $K_{2^k, 2^k}$'s we obtain that $ex^{*}(n, P_{2^k}) \ge (2^k)^2 \floor{n/2^{k+1}} = (1+o(1)) \frac{2^k}{2^k-1} ex(n, P_{2^k})$.

\vspace{2mm}

\textbf{Remark.} This construction provides a counterexample to the above mentioned conjecture of Keevash, Mubayi, Sudakov and Verstra\"ete \cite{KMSV}  whenever the largest clique that can be properly colored without a rainbow $P_{2^k}$ has size $2^k$. This is the case for $k = 2$, as noted before. The question of determining whether this is the case for any $k \ge 3$ remains an interesting open question (see \cite{Alon_P_S} for results in this direction). 

%Notice that this construction potentially provides many 
\vspace{2mm}

\textbf{Overview of the proof and organization.} Let $G$ be a graph which has a proper edge-coloring with no rainbow $P_{k+1}$. By induction on the length of the path, we assume there is a rainbow path $v_0v_1 \ldots v_k$ in $G$. Roughly speaking, we will show that the sum of degrees of the terminal vertices of the path, $v_0$ and $v_k$ is small. Our strategy is to find a set of distinct vertices $M:= \{a_1, b_1, a_2, b_2, \ldots, a_m, b_m\} \subseteq \{v_0, v_1, \ldots, v_k\}$ (whose size is as large as possible) such that for each $1 \le i \le m$, there is a rainbow path $P$ of length $k$ with $a_i$ and $b_i$ as terminal vertices and $V(P) = \{v_0, v_1, \ldots, v_k\}$; then we show that there are not many edges of $G$ incident to the vertices of $M$, which will allow us to delete the vertices of $M$ from $G$ and apply induction.  To this end, we define the set $T \subseteq \{v_0, v_1, \ldots, v_k\}$ as the set of all vertices $v \in \{v_0, v_1, \ldots, v_k\}$ where $v$ is a terminal vertex of some rainbow path $P$ with $V(P) = \{v_0, v_1, \ldots, v_k\}$; we call $T$ the set of \emph{terminal vertices}. We will then find $M$ as a subset of $T$; moreover, it will turn out that if the size of $T$ is large, then the size of $M$ is also large--therefore, the heart of the proof lies in showing that $T$ is large. 

In Section \ref{basic_claims}, we introduce the notation and prove some basic claims. Using these claims, in Section \ref{finding_terminal_vertices}, we will show that $T$ is large (i.e., that there are many terminal vertices). Then in Section \ref{finding_the_matching} we will find the desired subset $M$ of $T$ (which has few edges incident to it).

	\section{Proof of Theorem \ref{main_thm}}

Let $G$ be a graph on $n$ vertices, and suppose it has a proper edge-coloring $c: E(G) \to \mathbb{N}$ without a rainbow path of length $k+1$. Consider a longest rainbow path $P^*$ in $G$. We may suppose it is of length $k$, otherwise we are done by induction on $k$. For the base case $k =1$, notice that any path of length 2, has to be a rainbow path. Thus $G$ can contain at most $\frac{n}{2} < (\frac{9}{7}+2)n$ edges, so we are done.
	
	\subsection{Basic claims and Notation}
	\label{basic_claims}
	In the rest of the paper, the degree of a vertex $v \in V(G)$ be denoted by $d(v)$. 
	
%	Let $c: E(G) \to \mathbb{N}$ be the proper edge-coloring of $G$.  
%	Let $P^* = v_0v_1 \ldots v_{k}$. Suppose the color of the edge $v_{i-1}v_{i}$ is $c(v_{i-1}v_{i}) = c_i$ for each $1 \le i \le k$. We define some sets of colors corresponding to the path $P^*$. 
%	Let $L$ and $R$ denote the sets of colors of edges incident to $v_0$ and $v_k$ respectively. Notice that since the edges of $G$ are colored properly, we have $\abs{L} = d(v_0)$ and $\abs{R} = d(v_k)$. 
	
%	\vspace{2mm}
	
\begin{definition}
\label{BasicColorSets}
Let $P^* = v_0v_1 \ldots v_{k}$. Suppose the color of the edge $v_{i-1}v_{i}$ is $c(v_{i-1}v_{i}) = c_i$ for each $1 \le i \le k$. Let $L$ and $R$ denote the sets of colors of edges incident to $v_0$ and $v_k$ respectively. (Notice that since the edges of $G$ are colored properly, we have $\abs{L} = d(v_0)$ and $\abs{R} = d(v_k)$.)

We define the following subsets of $L$, $R$ and $\{c_1, c_2, \ldots, c_k\}$ corresponding to $P^*$. 
		
\begin{itemize}
			\item Let $L_{out}$ (respectively $R_{out}$) be the set of colors of the edges connecting $v_0$ (respectively $v_{k}$) to a vertex outside $P^*$. 
			
			Note that $L_{out} \subseteq \{c_1, c_2, \ldots, c_k\}$ and $R_{out} \subseteq \{c_1, c_2, \ldots, c_k\}$, otherwise we can extend $P^*$ to a rainbow path longer than $k$ in $G$. 
			
			\item Let $L_{in} = L \setminus L_{out}$ and $R_{in} = R \setminus R_{out}$. 
			
			\item 	Let $L_{old} = L \cap \{c_1, c_2, \ldots, c_k\}$ and $L_{new} = L\setminus \{c_1, c_2, \ldots, c_k\}$. Similarly, let $R_{old} = R \cap \{c_1, c_2, \ldots, c_k\}$, $R_{new} = R \setminus \{c_1, c_2, \ldots, c_k\}$.
			
			\item 	Let $S_L = \{c(v_{j-1}v_{j}) = c_j \mid v_0v_{j} \in E(G) \text{ and } c(v_{0}v_{j}) \in L_{new} \text{ and } 2 \le j \le k\}$ and $S_R = \{c(v_jv_{j+1}) = c_{j+1} \mid v_kv_{j} \in E(G) \text{ and } c(v_kv_{j}) \in R_{new} \text{ and } 0 \le j \le k-2\}$.
			
			Notice that $\abs{S_L} = \abs{L_{new}}$ and $\abs{S_R} = \abs{R_{new}}$.
			
			\item Let $L_{nice} = L \cap S_R$ and let $R_{nice} = R \cap S_L$. 
			
			\item Let $L_{res} = L_{in} \setminus (L_{new} \cup L_{nice}) = L_{old} \setminus (L_{nice}\cup L_{out})$, and $R_{res} = R_{in} \setminus (R_{new} \cup R_{nice}) = R_{old} \setminus (R_{nice}\cup R_{out})$.
		\end{itemize}
	\end{definition}

	%(respectively $R_{in}$) be the set of colors of the edges connecting $v_0$ (respectively $v_{l}$) to a vertex in $P^*$. Notice that $L = L_{in} \cup L_{out}$ and $R = R_{in} \cup R_{out}$.

	%(NOTE THAT HERE WE MAY NEED TO WRITE $\abs{S_R} = \abs{R_{new}}-1$ in case $v_0v_k$ exists).
	
	\begin{notation}
		For convenience, we let $\abs{L} = l$ and $\abs{R} = r$. 
		Moreover, let $\abs{L_{out}} = l_{out}, \abs{L_{old}} = l_{old}, \abs{L_{nice}} = l_{nice}, \abs{L_{new}} = l_{new}$ and $\abs{R_{out}} = r_{out}, \abs{R_{old}} = r_{old}, \abs{R_{nice}} = r_{nice}, \abs{R_{new}} = r_{new}$. 
	\end{notation}

	Note that $$d(v_0) = l_{in} + l_{out} = l_{new} + l_{old}$$ and $$d(v_k) = r_{in} + r_{out} = r_{new} + r_{old}.$$
	
	Now we prove some inequalities connecting the quantities defined in Definition \ref{BasicColorSets} for the path $P^*$.
	
	\begin{claim}
		\label{claimone}
		$L_{out} \cap S_R = \emptyset = R_{out} \cap S_L$. 
		This implies that $L_{out} \cap L_{nice} = \emptyset = R_{out} \cap R_{nice}$ (since $L_{nice} \subset S_R$ and $R_{nice} \subset S_L$).
	\end{claim}
	\begin{proof}
		Suppose for a contradiction that $L_{out} \cap S_R \not = \emptyset$. So there exists a vertex $w \not \in \{v_0, v_1, \ldots, v_k\}$ such that $c(v_kv_j) \in R_{new}$ and $c(wv_0) = c(v_jv_{j+1})$ for some $0 \le j \le k-2$. Consider the path $v_{j+1}v_{j+2} \ldots v_kv_jv_{j-1} \ldots v_0w$. The set of colors of the edges in this path is $\{c_1, c_2, \ldots, c_k\} \setminus \{c(v_jv_{j+1})\} \cup \{c(wv_0), c(v_kv_{j})\} = \{c_1, c_2, \ldots, c_k\} \cup \{c(v_kv_{j})\}$, so it is a rainbow path of length $k+1$ in $G$, a contradiction. 
		
		Similarly, by a symmetric argument, we have $R_{out} \cap S_L = \emptyset$. 
	\end{proof}

	\begin{claim}
		\label{claimzero}
		$l_{out} \le k - r_{new}$ and $r_{out} \le k - l_{new}$. 
	\end{claim}
	\begin{proof}
		By Claim \ref{claimone}, $L_{out} \cap S_R = \emptyset$. Since both $L_{out}$ and $S_R$ are subsets of $\{c_1, c_2, \ldots, c_k\}$, this implies, $\abs{L_{out}} = l_{out} \le k- \abs{S_R} = k - r_{new}$, as desired. Similarly, $r_{out} \le k - l_{new}$. 
	\end{proof}
	
	%Let the degree of a vertex $v \in V(G)$ be denoted by $d(v)$.
	%Note that $d(v_0) = l_{in} + l_{out} = l_{new} + l_{old}$ and $d(v_k) = r_{in} + r_{out} = r_{new} + r_{old}$.
	
	We will prove Theorem \ref{main_thm} by induction on the number of vertices $n$. 
	For the base cases, note that for all $n \le k$, the number of edges is trivially at most $$\binom{n}{2} \le \frac{kn}{2} < \left(\frac{9k}{7}+2\right)n,$$ so the statement of the theorem holds. If $d(v) < \frac{9k}{7}+2$ for some vertex $v$ of $G$, then we delete $v$ from $G$ to obtain a graph $G'$ on $n-1$ vertices. By induction hypothesis, the number of edges in $G'$ is less than $(\frac{9k}{7}+2)(n-1)$. So the total number of edges in $G$ is less than $(\frac{9k}{7}+2)n$, as desired. 
	
	Therefore, from now on, we assume that for all $v \in V(G)$, $$d(v) \ge \frac{9k}{7}+2.$$ Since $d(v_0) =  l = l_{old}+l_{new}$ and $l_{old} \le k$, we have that 
	\begin{equation}
	\label{lnew}
	l_{new} \ge \frac{2k}{7}+2.
	\end{equation}
	
	Similarly,
	\begin{equation}
	\label{rnew}
	r_{new} \ge \frac{2k}{7}+2.
	\end{equation}
	
	\begin{claim}
		\label{niceclaim}
		We have $$l_{nice} + r_{nice} \ge \frac{4k}{7}+4.$$
	\end{claim}
	\begin{proof}
		First notice that $L_{res} \cap S_R = \emptyset$. Indeed, by definition, $L_{res} \cap S_R = (L_{res} \cap L) \cap S_R = L_{res} \cap (L \cap S_R) = L_{res} \cap L_{nice} = \emptyset$. Moreover, by Claim \ref{claimone}, $L_{out} \cap S_R = \emptyset$. Therefore, we have $(L_{res} \cup L_{out}) \cap S_R = \emptyset$. Moreover, $(L_{res} \cup L_{out}) \cup S_R \subseteq \{c_1,c_2, \ldots, c_k\}$. Therefore, $l_{res} + l_{out} \le k - \abs{S_R} = k - r_{new}$. On the other hand, by definition, $l_{res} + l_{out}= l-l_{new}-l_{nice}$. So $$l-l_{new}-l_{nice} \le k - r_{new}.$$ By a symmetric argument, we get $$r-r_{new}-r_{nice} \le k - l_{new}.$$ Adding the above two inequalities and rearranging, we get $l + r -l_{nice} -r_{nice} \le 2k$, so $$l_{nice} + r_{nice} \ge l + r- 2k = d(v_0) + d(v_k)-2k \ge \frac{4k}{7}+4,$$ as required.
		
	\end{proof}

\subsection{Finding many terminal vertices}
\label{finding_terminal_vertices}
	
	\begin{definition} [Set of terminal vertices]
		% Consider all the rainbow paths of length $k$ whose vertex set is $N = \{v_0, v_1, \ldots, v_k\}$. 
		Let $T$ be the set of all vertices $v \in \{v_0, v_1, v_2, \ldots, v_k\}$ such that $v$ is a terminal (or end) vertex of some rainbow path $P$ with $V(P) = \{v_0, v_1, v_2, \ldots, v_k\}$.

		For convenience, we will denote the size of $T$ by $t$.
	\end{definition}
	
	The next lemma yields a lower bound on the number of terminal vertices and is crucial to the proof of Theorem \ref{main_thm}.
	
	\begin{lemma}
		\label{terminal_vertices}
		We have $$\abs{T} = t \ge \frac{3k}{7}+1.5.$$
	\end{lemma}

The rest of this subsection is devoted to the proof of Lemma \ref{terminal_vertices}. 

	\subsubsection*{Proof of Lemma \ref{terminal_vertices}}
	
	Recall that $P^* = v_0 v_1\ldots v_k$ and $c(v_jv_{j+1}) = c_j$. First we make a simple observation. 
	
	\begin{observation}
		\label{vovljumping}
		If $c(v_0v_k) \in L_{new} \cup R_{new}$, then every vertex $v_i \in T$. Indeed, the path $v_iv_{i-1}v_{i-2} \ldots v_0v_k v_{k -1} \ldots v_{i+1}$ is a rainbow path with $v_i$ as a terminal vertex. Thus $\abs{T} = k+1 \ge \frac{3k}{7}+1.5$, and we are done. So from now on, we assume $c(v_0v_k) \not \in L_{new} \cup R_{new}$.
		
		This implies that $c(v_0v_1) \not \in L_{nice}$ and $c(v_kv_{k -1}) \not \in R_{nice}$, because $c(v_0v_1) \not \in S_R$ and $c(v_kv_{k -1}) \not \in S_L$. 
	\end{observation}

	\begin{claim}
		\label{new_edge}
		If $v_0v_i$ is an edge such that $c(v_0v_i) \in L_{new}$ then $v_{i-1} \in T$.
	\end{claim}
	\begin{proof}
		Consider the path $v_{i-1}v_{i-2} \ldots v_0 v_i v_{i+1} \ldots v_k$. Clearly it is a rainbow path of length $k$ in which $v_{i-1}$ is a terminal vertex.
	\end{proof}
	
	% \begin{claim}
	% If $v_0v_i$ is an edge such that $c(v_0v_i) \in L_{nice}$, then either $v_{i-1} \in T$ or $v_{i+1} \in T$.
	% \end{claim}
	% 
	%NOTE THAT we may need to change the range of j to start from 0 if we change the definition of SR
	Suppose $v_0v_i$ is an edge such that $c(v_0v_i) \in L_{nice}$. Since $c(v_0v_k) \not \in R_{new}$, by the definition of $L_{nice}$, there exists an integer $j$ (with $1 \le j \le k-2$) such that $c(v_kv_j) \in R_{new}$ and $c(v_0v_i) = c(v_jv_{j+1}) =c_j$.

	\begin{claim}
		\label{nice_edges}
		If $c(v_0v_i) \in L_{nice}$ then $v_{i-1} \in T$ or $v_{i+1} \in T$.
		
		Moreover, let $j$ be an integer (with $1 \le j \le k-2$) such that $c(v_kv_j) \in R_{new}$ and $c(v_0v_i) = c(v_jv_{j+1}) =c_j$. 
		
		If $j \ge i$, then $v_{i-1} \in T$, and if $j < i$ then $v_{i+1} \in T$.
	\end{claim}
	\begin{proof}
		Observe that since $c(v_0v_i) \in L_{nice} \subset S_R$, we have that $c(v_kv_j) \in R_{new}$ (by definition of $S_R$).
		
		First let $j \ge i$. In this case consider the path $v_{i-1}v_{i-2}\ldots v_0 v_iv_{i+1}\ldots v_jv_kv_{k -1} \ldots v_{j+1}$. It is easy to see that the set of colors of the edges in this path is $\{c_1, c_2, \ldots, c_k\} \setminus \{c_i\} \cup \{c(v_jv_k)\}$. As $c(v_jv_k) \in R_{new}$, the path is rainbow with $v_{i-1}$ as a terminal vertex. So $v_{i-1} \in T$.
		
		If $j < i$, then consider the path $v_{j+1}v_{j+2}\ldots v_iv_0v_1 \ldots v_j v_k v_{k -1} \ldots v_{i+1}$. It is easy to see that the set of colors of the edges in this path is $\{c_1, c_2, \ldots, c_k\} \setminus \{c_{i+1}\} \cup \{c(v_jv_k)\}$, so the path is rainbow again, with $v_{i+1}$ as a terminal vertex. So $v_{i+1} \in T$.
	\end{proof}
	
	\begin{definition}
		\label{aa'}
		Let $b$ be the largest integer such that $c(v_0v_b) \in L_{new}$ and there exists $b' > b$ with $c(v_0v_b') \in L_{new}$. (In other words, $b$ is the second largest and $b'$ is the largest integer $j$ such that $c(v_0v_j) \in L_{new}$.)
		Let $a$ be the smallest integer such that $c(v_kv_a) \in R_{new}$ and there exists $a' < a$ with $c(v_kv_a') \in R_{new}$. (That is, $a$ is the second smallest and $a'$ is the smallest integer $j$ such that $c(v_kv_j) \in R_{new}$.)
	\end{definition}

	\begin{notation}
		For any integers, $0 \le x \le y \le k$, let
		$$T^{x,y} = \{v_i \in T \mid x \le i \le y\},$$
		and $\abs{T^{x,y}} = t^{x,y}$.
		
		Notice that $t = t^{0,k} = 2 + t^{1, k-1}$, as $v_0$ and $v_k$ are both terminal vertices. 
	\end{notation}

Now we will show that if $a > b$, then Lemma \ref{terminal_vertices} holds. Suppose $a > b$. Then by the definition of $a$ and $b$, we have $$\abs{\{ i \mid 2 \le i \le b \text{ and } c(v_0v_i) \in L_{new} \}} = \abs{L_{new}}-1 = l_{new}-1.$$ By Claim \ref{new_edge}, we know that whenever $ c(v_0v_i) \in L_{new}$, we have $v_{i-1} \in T$. This shows that $t^{1,b-1} \ge  l_{new}-1.$ Similarly, by a symmetric argument, we get $t^{a+1,k-1} \ge  r_{new}-1.$ Therefore, $$t = 2+ t^{1, k-1} = 2+ t^{1, b-1} + t^{b,a} + t^{a+1,k-1} \ge 2 + (l_{new}-1) + ( r_{new}-1) = l_{new} + r_{new}.$$

Now using \eqref{lnew} and \eqref{rnew}, we have 
$$t = l_{new} + r_{new} \ge \frac{2k}{7}+2 + \frac{2k}{7}+2 = \frac{4k}{7} + 4,$$ proving Lemma \ref{terminal_vertices}. Therefore, from now on, we always assume $a \le b$. 
	
	\begin{claim}
		\label{bothsides_new}
		If $c(v_0v_i) \in L_{new}$ or $c(v_kv_i) \in R_{new}$, and $a \le i \le b$, then $v_{i-1} \in T$ and $v_{i+1} \in T$.
	\end{claim}
	\begin{proof}[Proof of Claim]
		First suppose $c(v_0v_i) \in L_{new}$. Then by Claim \ref{new_edge}, $v_{i-1} \in T$. We want to show that $v_{i+1} \in T$.
		
		Observe that if $i = a$, then by Claim \ref{new_edge} again, we have $v_{i+1} \in T$ because $v_kv_i \in R_{new}$. So let us assume $a < i$ and show that $v_{i+1} \in T$. Notice that there exists $a^* \in \{a, a'\}$ (see Definition \ref{aa'} for the definition of $a$ and $a'$) such that $c(v_0v_i) \not = c(v_{a^*}v_k)$. Now consider the path $v_{a^*+1}v_{a^*+2} \ldots v_i v_0 v_1 \ldots v_{a^*}v_k v_{k -1} \ldots  v_{i+1}$. The set of colors of the edges in this path are $\{c_1, c_2, \ldots, c_k\} \setminus \{c_{a^*+1}, c_{i+1}\} \cup \{c(v_0v_i), c(v_{a^*}v_k)\}$, and it is easy to check that all the colors are different, so the path is rainbow with $v_{i+1}$ as a terminal vertex. 
		
		Now suppose $c(v_kv_i) \in R_{new}$. Then a similar argument shows that $v_{i-1} \in T$ and $v_{i+1} \in T$ again, completing the proof of the claim.
	\end{proof}
	
	Now we introduce some helpful notation. 
	
	\begin{notation}
		For any integers, $0 \le x \le y \le k$, let 
		$$L_{nice}^{x,y} = \{c(v_0v_i) \in L_{nice} \mid x \le i \le y\},$$ 
		$$R_{nice}^{x,y} = \{c(v_kv_i) \in R_{nice} \mid x \le i \le y\},$$ $$L_{new}^{x,y} = \{c(v_0v_i) \in L_{new} \mid x \le i \le y\},$$ $$R_{new}^{x,y} = \{c(v_kv_i) \in R_{new} \mid x \le i \le y\},$$ 
		$$T^{x,y} = \{v_i \in T \mid x \le i \le y\}.$$
		
		Moreover, let $\abs{L_{nice}^{x,y}} = l_{nice}^{x,y}$, $\abs{R_{nice}^{x,y}} = r_{nice}^{x,y}$, $\abs{L_{new}^{x,y}} = l_{new}^{x,y}$,  $\abs{R_{new}^{x,y}} = r_{new}^{x,y}$. 	
	\end{notation}

	Note that by definition of $a$ and $b$, $l_{new} = l_{new}^{0,a-1} + l_{new}^{a,b} + 1$ and $r_{new} = 1+ r_{new}^{a,b} + r_{new}^{b+1,l}$. Using Observation \ref{vovljumping}, for any integer $z$, we have the following:
	\begin{equation}
	\label{0to2nice}
	L_{nice}^{0,z}  = L_{nice}^{2,z} \text{ and } R_{nice}^{z, k} = R_{nice}^{z, k-2}.
	\end{equation}
	Moreover, by definition of $L_{new}$ and $R_{new}$, we have
	\begin{equation}
	\label{0to2new}
	L_{new}^{0,z}  = L_{new}^{2,z} \text{ and } R_{new}^{z,k}  = R_{new}^{z,k-2}.
	\end{equation}
	
Informally speaking, Claim \ref{nice_edges} and Claim \ref{bothsides_new} assert that each edge $e = v_0v_i$ such that $c(v_0v_i) \in L_{new} \cup L_{nice}$ ``creates" a terminal vertex $x = v_{i-1} \in T$ or $ x = v_{i+1} \in T$ (or sometimes both). Similarly, each edge $e = v_kv_i$ such that $c(v_kv_i) \in R_{new} \cup R_{nice}$ ``creates" a terminal vertex $x = v_{i-1} \in T$ or $x = v_{i+1} \in T$ (or both). In the next two claims, by double counting the total number of such pairs $(e,x)$, we prove lower bounds on the number of terminal vertices in different ranges (i.e., $t^{0,a-1}, t^{b+1,k}$ and $t^{a,b}$), in terms of $l_{new}, r_{new}, l_{nice}$ and $r_{nice}$.
	
	\begin{claim}
		\label{0toa-1andb+1tok}
		We have,
		$$t^{0,a-1} \ge \frac{1}{2} \left(l_{nice}^{0,a}+l_{new}^{0,a}+\frac{r_{nice}^{0,a}}{2}\right),$$
		and 
	   $$t^{b+1,k} \ge \frac{1}{2} \left(r_{nice}^{b,k}+r_{new}^{b,k}+\frac{l_{nice}^{b,k}}{2}\right). $$
	\end{claim}
	
\begin{proof}[Proof of Claim]
		%\item Let $T^{0,a-1}$ denote the set of terminal vertices in the set $\{v_0, v_1, \ldots, v_{a-1}\}$. In other words, $T^{0,a-1} = T \cap \{v_0, v_1, \ldots, v_{a-1}\}$, and let $t^{0,a-1} = \abs{T^{0,a-1}}$. We will find a lower bound on $t^{0,a-1}$. 
		By Claim \ref{nice_edges}, and by the fact that there is only one $j$ such that $c(v_kv_j) \in R_{new}^{0,a-1}$, it is easy to see that for all but at most one $i$, we have the following: if $c(v_0v_i) \in L_{nice}^{0,a} = L_{nice}^{2,a}$ (equality here follows from \eqref{0to2nice}), then $v_{i-1} \in T^{1, a-1}$.  So there are at least $l_{nice}^{2,a}-1$ pairs $(v_0v_i,x)$ such that $c(v_0v_i) \in L_{nice}^{2,a}$ and $x = v_{i-1} \in T^{1, a-1}$.
		
		If $c(v_0v_i) \in L_{new}^{0,a} = L_{new}^{2,a}$ (equality here follows from \eqref{0to2new}), then by Claim \ref{new_edge}, $v_{i-1} \in T^{1,a-1}$. So there are $l_{new}^{2,a}$ pairs $(v_0v_i,x)$ such that $c(v_0v_i) \in L_{new}^{2,a}$ and $x = v_{i-1} \in T^{1,a-1}$. 
		
		Adding the previous two bounds, the total number of pairs $(v_0v_i,x)$ such that $c(v_0v_i) \in L_{nice}^{0,a} \cup  L_{new}^{0,a} =  L_{nice}^{2,a} \cup  L_{new}^{2,a}$ and $x = v_{i-1} \in T^{1,a-1}$, is at least $l_{nice}^{2,a}-1+l_{new}^{2,a}$. This implies $t^{1,a-1} \ge l_{nice}^{2,a}-1+l_{new}^{2,a}$. Therefore, using that $v_0$ is also a terminal vertex, we have 
		\begin{equation}
		\label{eq1}
		t^{0,a-1} \ge l_{nice}^{2,a}+l_{new}^{2,a}.
		\end{equation}

		If $c(v_kv_i) \in R_{nice}^{0,a-1}$, then by Claim \ref{nice_edges}, there is a vertex $x \in \{v_{i-1}, v_{i+1}\}$ such that $x \in T$. So the number of  pairs $(v_kv_i, x)$ such that  $c(v_kv_i) \in R_{nice}^{0,a-1}$, $x \in \{v_{i-1}, v_{i+1}\}$ and $x \in T$, is at least $r_{nice}^{0,a-1}$. By the pigeon-hole principle, either the number of  pairs $(v_kv_i, v_{i-1})$ with $c(v_kv_i) \in R_{nice}^{0,a-1}$, $v_{i-1} \in T$, or the number of pairs $(v_kv_i, v_{i+1})$ with $c(v_kv_i) \in R_{nice}^{0,a-1}$,$v_{i+1} \in T$, is at least $r_{nice}^{0,a-1}/2$. In the first case, we get $t^{0, a-2} \ge r_{nice}^{0,a-1}/2$ and in the second case, we get $t^{1, a} \ge r_{nice}^{0,a-1}/2$. As $t^{0,a-1} \ge t^{0, a-2}$ and $t^{0,a-1} \ge t^{1, a}$, in both cases we have, 
		\begin{equation}
		\label{eq2}
		t^{0,a-1} \ge \frac{r_{nice}^{0,a-1}}{2}.
		\end{equation}
		
		Therefore, adding up \eqref{eq1} and \eqref{eq2}, we get
		\begin{equation*}
		\label{eq3}
		2t^{0,a-1} \ge l_{nice}^{2,a}+l_{new}^{2,a}+\frac{r_{nice}^{0,a-1}}{2} = l_{nice}^{0,a}+l_{new}^{0,a}+\frac{r_{nice}^{0,a}}{2}. 
		\end{equation*}
		Note that the equality follows from \eqref{0to2nice} and the fact that $r_{nice}^{0,a-1} = r_{nice}^{0,a}$ because $c(v_kv_a) \in R_{new}.$ By a symmetric argument, we have
		
		\begin{equation*}
		\label{eq4}
		2t^{b+1,k} \ge r_{nice}^{b,k-2}+r_{new}^{b,k-2}+\frac{l_{nice}^{b+1,k}}{2} = r_{nice}^{b,k}+r_{new}^{b,k}+\frac{l_{nice}^{b,k}}{2}. 
		\end{equation*}
		This finishes the proof of the claim.
	\end{proof}
		
		Now we prove a lower bound on $t^{a,b}$.
		
		\begin{claim}
			\label{atob}
			$$t^{a,b} \geq \frac{1}{4} \left( l_{nice}^{a+1, b-1} + r_{nice}^{a+1, b-1}+ 2(l_{new}^{a+1, b}+ r_{new}^{a, b-1}) -2 \right).$$
		\end{claim}
		
		\begin{proof}[Proof of Claim]
		%Suppose $e \in \{v_0v_i, v_kv_i\}$. We will define a set $S$ of pairs $(e,x)$ of edges $e$ and vertices $x$.
		
		Let us construct a set $S$ of pairs $(e,x)$ such that $e \in L_{in} \cup R_{in}$ and $x \in T$ with certain properties. 
		
		If $c(e) \in L_{nice}^{a+1,b-1} \cup R_{nice}^{a+1, b-1}$, then by Claim \ref{nice_edges}, there is a vertex $x \in \{v_{i-1}, v_{i+1}\}$ such that $x \in T$ (in particular, $x \in T^{a,b}$). Add all such pairs $(e,x)$ to $S$. Therefore, the number of pairs $(e,x)$ added to $S$ so far, is $l_{nice}^{a+1, b-1} + r_{nice}^{a+1, b-1}$.
		
		%So the number of pairs $(e,x)$ added to $S$ such that $e \in \{v_0v_i, v_kv_i\}$, $c(e) \in L_{nice}^{a+1,b-1} \cup R_{nice}^{a+1, b-1}$, and $x \in T^{a,b}$ is at least $l_{nice}^{a+1, b-1} + r_{nice}^{a+1, b-1}$.  Add all these pairs $(e, x)$ to the set $S$. 
		
		For each $e$ such that $c(e) \in L_{new}^{a+1, b} \cup R_{new}^{a, b-1}$, we have both $v_{i-1}, v_{i+1} \in T$ by Claim \ref{bothsides_new}; we add both the pairs $(e, v_{i-1})$ and $(e, v_{i+1})$ to $S$. Therefore the number of pairs $(e,x)$ added to $S$ in this step is $2(l_{new}^{a+1, b}+ r_{new}^{a, b-1})$. Thus, $$\abs{S} = l_{nice}^{a+1, b-1} + r_{nice}^{a+1, b-1}+ 2(l_{new}^{a+1, b}+ r_{new}^{a, b-1}).$$ Note that except the pairs $(v_kv_b, v_{b+1})$, $(v_0v_a, v_{a-1})$, all other pairs $(e,x)$ in $S$ are such that $x \in T^{a,b}$. Moreover, for each $x \in T^{a,b}$, there are at most four pairs $(e,x)$ in $S$. Therefore, we have
		
		\begin{equation*}
		\label{eq5}
		4t^{a,b} \geq \abs{S}-2 \ge l_{nice}^{a+1, b-1} + r_{nice}^{a+1, b-1}+ 2(l_{new}^{a+1, b}+ r_{new}^{a, b-1}) -2,
		\end{equation*}
		finishing the proof of the claim.
	\end{proof}
%	\end{itemize}
	%except the pairs $(v_kv_b, v_{b+1})$ and $(v_0v_a, v_{a-1})$. 
	
	% \item Suppose $a+1 \le i \le k-1$. If $c(v_kv_i) \in R_{nice} \cup R_{new}$, then add the pair $(v_kv_i, v_{i+1})$ to $S$. If $c(v_0v_i) \in L_{nice}$, then by Claim \ref{nice_edges}, there is a vertex $x \in \{v_{i-1}, v_{i+1}\}$ such that $x \in T$. Add the pair $(v_0v_i, x)$ to $S$. 
	% 
	
	% (By Claim \ref{nice_edges}, for all but at most one $i$ with $1 \le i \le a$, if $c(v_0v_i) \in L_{nice}$, then $v_{i-1} \in T$. Add all these pairs $(v_0v_i, v_{i-1})$ to $S$.)
	
	By Claim \ref{0toa-1andb+1tok} and Claim \ref{atob}, we have 
	$$  2(2t^{0,a-1} + 2t^{b+1,l})+ 4t^{a,b} \ge  2\left(l_{nice}^{0,a}+l_{new}^{0,a}+\frac{r_{nice}^{0,a}}{2} + r_{nice}^{b,k}+r_{new}^{b,k}+\frac{l_{nice}^{b,k}}{2}\right)$$ $$+ l_{nice}^{a+1, b-1}+ r_{nice}^{a+1, b-1}+ 2(l_{new}^{a+1, b}+ r_{new}^{a, b-1}) -2   $$
	This implies, 
	$$4t \ge l_{nice} + r_{nice} + 2l_{new}^{0,b} + 2r_{new}^{a, l}+ l_{nice}^{0,a}+r_{nice}^{b,l}-2.$$
	By the definition of $a$ and $b$, $l_{new}^{0,b} = l_{new} -1$ and $r_{new}^{a, l} = r_{new} - 1$. So, we get
	$$4t \ge l_{nice} + r_{nice} + 2l_{new} + 2r_{new}+ l_{nice}^{0,a}+r_{nice}^{b,l}-6$$ $$ \ge l_{nice}+r_{nice} + 2(l_{new} + r_{new})-6. $$
	
	Now by Claim \ref{niceclaim} and inequalities \eqref{lnew} and \eqref{rnew}, we get that $$4t \ge \frac{4k}{7}+4 + 2 \left(\frac{2k}{7}+2 + \frac{2k}{7}+2 \right) -6 = \frac{12k}{7}+6.$$ Therefore, $$ t \ge \frac{3k}{7}+1.5,$$ completing the proof of Lemma \ref{terminal_vertices}.

	\subsection{Finding a large subset of vertices with few edges incident to it}
\label{finding_the_matching}

	Now we define an auxiliary graph $H$ with the vertex set $V(H) = T$ and edge set $E(H)$ such that $ab \in E(H)$ if and only if there is a rainbow path $P$ in $G$ with $a$ and $b$ as its terminal vertices and $V(P) = V(P^*) = \{v_0, v_1, \ldots, v_k\}$. 
	
	\begin{claim}
		\label{min_degree_of_H}
		%Let $v$ be a non-isolated vertex of the graph $A$. Then degree of $v$ in $A$ is at least $2k/7$.
		The degree of every vertex $u$ in $H$ is at least $\frac{2k}{7}+2$.
	\end{claim}
	\begin{proof}[Proof of Claim]
		As $u \in V(H) = T$, $u$ is a terminal vertex. So there is a rainbow path $P = u_0u_1 \ldots u_k$ in $G$ such that $u_0 =u$ and $\{u_0, u_1, \ldots, u_k \} = \{v_0, v_1, \ldots, v_k \}$. We define the sets $L, R, L_{new}, R_{new}$ corresponding to $P$ in the same way as we did for $P^*$ (in Definition \ref{BasicColorSets}). Moreover, since $P^*$ was defined as an arbitrary rainbow path of length $k$, \eqref{rnew} holds for $P$ as well -- i.e., $\abs{R_{new}} = r_{new} \ge  \frac{2k}{7}+2$. We claim that if $u_ku_j$ is an edge in $G$ such that $c(u_ku_j) \in R_{new}$, then $uu_{j+1} \in E(H)$. Indeed, consider the path $u_0u_1\ldots u_ju_ku_{k -1} \ldots u_{j+1}$. This is clearly a rainbow path with terminal vertices $u = u_0$ and $u_{j+1}$. So $u$ and $u_{j+1}$ are adjacent in $H$, as required. This shows that degree of $u$ in $H$ is at least $r_{new} \ge \frac{2k}{7}+2$, as desired.
	\end{proof}
	
	Size of a matching is defined as the number of edges in it. The following proposition is folklore.
	
	\begin{prop}
		\label{matching}
		Any graph $G$ with minimum degree $\delta(G)$ has a matching of size 
		$$\min \left \{\delta(G), \floor{\frac{V(G)}{2}} \right \}.$$
	\end{prop}
	
	We know that $\delta(H) \ge \frac{2k}{7}+2$ by Claim \ref{min_degree_of_H}. Moreover $\abs{V(H)} = \abs{T} = t$. So applying Proposition \ref{matching} for the graph $H$ and using Lemma  \ref{terminal_vertices}, we obtain that the graph $H$ contains a matching $M$ of size 
	
	\begin{equation}
	\label{matching_size}
	m := \min \left \{\frac{2k}{7}+2, \floor{\frac{t}{2}}\right \} \ge \frac{3k}{14}.
	\end{equation}
	
	Let the edges of $M$ be $a_1b_1, a_2b_2, \ldots, a_mb_m$.  
	Moreover, let $$n_i = \abs{\{xy \mid xy \not \in E(G), x \in \{a_i, b_i\} \text{ and } y \in \{v_0, v_1, v_2, \ldots, v_k\} \setminus \{a_i, b_i\} \}}.$$
	
	\begin{claim}
		\label{gm}
		The number of edges in the subgraph of $G$ induced by $M$ is $$\abs{E(G[M])} \ge \binom{2m}{2} - \left(\sum_{i=1}^m \frac{n_i}{2} + m\right) = 2m^2-2m-\sum_{i=1}^m \frac{n_i}{2}.$$
	\end{claim}
	\begin{proof}[Proof of Claim]
		Note that the sum $\sum_i n_i$ counts each pair $xy \not \in E(G)$ with $x,y \in V(M)$ exactly twice unless $xy = a_ib_i$ for some $i$. Therefore, the number of pairs $xy \not \in E(G)$ in the subgraph of $G$ induced by $M$ is at most $\sum_i \frac{n_i}{2} + m$. Thus the number of edges of $G$ in the subgraph induced by $M$ is at least $\binom{2m}{2} - (\sum_i \frac{n_i}{2} + m)$, which implies the desired claim.
	\end{proof}
	
	\begin{claim}
		\label{aibi}
		The sum of degrees of $a_i$ and $b_i$ in $G$ is at most $3k - \frac{n_i}{2}$.
	\end{claim}
	\begin{proof}[Proof of Claim]
		Since $a_ib_i$ is an edge in the auxiliary graph $H$, there is a rainbow path $P = u_0u_1 \ldots u_k$ in $G$ such that $u_0 =a_i$, $u_k = b_i$ and $\{u_0, u_1, \ldots, u_k \} = \{v_0, v_1, \ldots, v_k \}$. We define the sets $L, R, L_{in}, L_{out}, R_{in}, L_{new}, R_{new}$ corresponding to $P$ in the same way as we did for $P^*$ (in Definition \ref{BasicColorSets}). Therefore, degree of $a_i$ is $l \le l_{new} + k$. Similarly, degree of $b_i$ is at most $r_{new} + k$. So the sum of degrees of $a_i$ and $b_i$ in $G$ is at most 
		\begin{equation}
		\label{one1}
		2k + l_{new} + r_{new}.
		\end{equation}
		
		On the other hand, the sum of degrees of $a_i$ and $b_i$ in $G$ is $l + r = l_{in} + l_{out} + r_{in} + r_{out}$. By   Claim \ref{claimzero}, this is at most  $(l_{in} + r_{in}) + k - r_{new} + k - l_{new} = (l_{in} + r_{in}) + 2k - l_{new} - r_{new}$. Moreover, it is easy to see that $l_{in} + r_{in} \le 2k -n_i$ by the definition of $n_i$. Therefore, the sum of degrees of $a_i$ and $b_i$ in $G$ is at most 
		\begin{equation}
		\label{two2}
		2k -n_i + 2k - l_{new} - r_{new}.
		\end{equation}
		Adding up \eqref{one1} and \eqref{two2} and dividing by 2, we get that the sum of degrees of $a_i$ and $b_i$ in $G$ is at most $$\frac{(2k + 2k -n_i + 2k)}{2} = \frac{(6k -n_i)}{2} = 3k -\frac{n_i}{2},$$ as desired.
	\end{proof}
	
	The sum $\sum_{i=1}^m (d(a_i) +d(b_i))$ counts each edge in the subgraph of $G$ induced by $M$ exactly twice (note that here $d(v)$ denotes the degree of the vertex $v$ in $G$). Therefore, the number of edges of $G$ incident to the vertices of $M$ is at most $\sum_{i=1}^m (d(a_i) +d(b_i)) - \abs{E(G[M])}$. Now using Claim \ref{gm} and Claim \ref{aibi}, the number of edges of $G$ incident to the vertices of $M$ is at most $$\sum_{i=1}^m \left(3k - \frac{n_i}{2}\right) - \left(2m^2-2m-\sum_{i=1}^m \frac{n_i}{2}\right) = 3km - 2m^2+2m = (3k+2-2m)m.$$ Now by \eqref{matching_size}, this is at most $$(3k+2-2m)m \le \left(3k+2-2 \left(\frac{3k}{14}\right)\right)m = \left(\frac{9k}{7}+1\right)2m < \left(\frac{9k}{7}+2\right)2m.$$
	%floor{t/2} \ge 3k/14 - 1
	We may delete the vertices of $M$ from $G$ to obtain a graph $G'$ on $n-2m$ vertices. By induction hypothesis, $G'$ contains less than $(\frac{9k}{7}+2)(n-2m)$ edges. Therefore, $G$ contains less than $$\left(\frac{9k}{7}+2\right)2m + \left(\frac{9k}{7}+2\right)(n-2m) =\left(\frac{9k}{7}+2\right)n$$ edges, as desired. This completes the proof of Theorem \ref{main_thm}.

	\section*{Acknowledgements}
	The research of the authors is partially supported by the National Research, Development and Innovation Office NKFIH, grant K116769.


\begin{thebibliography}{10}
		
		\bibitem{Alon_P_S}
		N. Alon, A. Pokrovskiy and B. Sudakov. Random subgraphs of properly edge-coloured complete graphs and long rainbow cycles. \emph{Israel Journal of Mathematics} (2017) \textbf{222.1},  317--331.
			
			
		\bibitem{Bose_Chowla} R.C. Bose and S. Chowla. Theorems in the additive theory of numbers. \emph{Comment. Math. Helv.} (1962/1963) \textbf{37},  141–-147.
		
		\bibitem{Das}
		S. Das, C. Lee and B. Sudakov. Rainbow Tur\'an problem for even cycles. \emph{European Journal of Combinatorics} (2013)  \textbf{34}, 905--915.
		
		\bibitem{ErdGallai}
		P. Erd\H{o}s and T. Gallai. On maximal paths and circuits of graphs. \emph{Acta Mathematica Academiae Scientiarum Hungaricae} (1959) \textbf{10}, 337-–356.
		
		\bibitem{JPalmer}
		D. Johnston, C. Palmer and A. Sarkar. Rainbow Tur\'an Problems for Paths and Forests of Stars. \emph{The Electronic Journal of Combinatorics} (2017) \textbf{24(1)}, 1--34.
		
		\bibitem{KMSV}
		P. Keevash, D. Mubayi, B. Sudakov and J. Verstra\"ete. Rainbow Tur\'an problems. \emph{Combinatorics, Probability and Computing}  (2007)  \textbf{16}, 109--126.
		
		\bibitem{Maamoun}
		M. Maamoun and H. Meyniel. On a problem of G. Hahn about coloured Hamiltonian paths in $K_{2^t}$. \emph{Discrete Mathematics} (1984)  \textbf{51}, 213--214.
		
	\end{thebibliography}
\end{document}